\documentclass[a4,12pt,reqno]{article}
\usepackage{amscd,amsmath,amssymb,amsfonts,amsthm}
\usepackage[arrow,matrix]{xy}
\usepackage{graphicx}
\usepackage{color}
\usepackage[utf8]{inputenc}
\usepackage[english]{babel}
\theoremstyle{plain}
\newtheorem{theorem}{Theorem}[section]
\newtheorem{lemma}{Lemma}[section]

\theoremstyle{definition}

\newtheorem{defn}{Definition}[section]

\begin{document}

\title{On extremal results of multiplicative Zagreb indices of trees with given distance $k$-domination number}
\author{Fazal Hayat\\
School of Mathematical Sciences, South China Normal University,\\
 Guangzhou 510631, PR China\\
 E-mail: fhayatmaths@gmail.com}

 \date{}
\maketitle

\begin{abstract}  The first multiplicative Zagreb index $\Pi_1$ of a graph $G$ is the product of the square of every vertex degree, while the second multiplicative Zagreb index $\Pi_2$ is the product of the products of degrees of pairs of adjacent vertices. In this paper, we give sharp lower bound for $\Pi_1$ and upper bound for $\Pi_2$ of trees with given distance $k$-domination number, and characterize those trees attaining the bounds.
\end{abstract}

\section{Introduction}

In this article we consider only simple, undirected and connected graphs. Let $G$ be a graph with vertex set $V(G)$ and
edge set $E(G)$. The degree of $v\in V(G)$, denoted by $d_G(v)$, is the number of vertices in $G$ adjacent to $v$, and the neighborhood of $v$ is the set $N_G(v) = \{w\in V(G): vw \in E(G)\}$. Evidently, $|N_G(v)|=d_G(v)$. A vertex with degree one is called pendent vertex.  The distance between any two vertices $u$ and $v$ of a graph $G$ is denoted by $d_G(u,v)$.  The maximum distance from a vertex $v\in V(G)$ to all other vertices of $G$ is called eccentricity of $v$ in  $G$. The diameter of a graph $G$ is the maximum eccentricity of all vertices of $G$.

A graph $G$ that has $n$ vertices and $n-1$ edges is called a tree. As usual, by $P_n$ and $S_n$ we denote the path and the star on $n$ vertices, respectively.


The first and the second Zagreb indices are among
 the oldest topological molecular descriptors, see \cite{7}. They are defined as follows:
\[
 M_1(G)=\sum_{u\in V(G)}d(u)^2 \mbox{ and }   M_2(G)=\displaystyle\sum_{uv\in E(G)}d(u)d(v).
\]
Many interesting properties of them may be found in \cite{8,10, 9, 12,Z,ZG}.

In 2010, Todeschini et al. \cite{13, 14} put forward  the multiplicative Zagreb indices  as follows:
\[
\Pi_1(G)=\prod_{u\in V(G)}d(u)^2 \mbox{ and } \Pi_2(G)=\prod_{uv\in E(G)}d(u){d(v)}.
\]
It is easily seen that $\Pi_2(G)=\prod\limits_{uv\in E(G)}d(u){d(v)}=\prod\limits_{u\in V(G)}d(u)^{d(u)}$.
Some properties for the  multiplicative Zagreb indices have been established, see \cite{2,16,17,18,15,19}.

For a positive integer $k$, a set $D\subseteq V(G)$  is said to be distance $k$-dominating set of $G$ if for every vertex  $u\in V(G)\setminus D$, $d_G(u,v)\le k$ for some vertex $v\in D$.  The minimum cardinality among all distance $k$-dominating set of $G$ is called the distance $k$-domination number of $G$, denoted by $\gamma_k(G)$. A distance $1$-dominating set of $G$ is known as a dominating set of $G$ and the distance $1$-domination number of $G$ is just the classical domination number of $G$.


Borovicanin and Furtula \cite{24} presented sharp lower and upper bounds on Zagreb indices of trees in terms of domination number,  and  Wang et al. \cite{1} found  sharp lower and upper  bounds on  multiplicative Zagreb indices of trees in terms of domination number. Recently, Pei and Pan \cite{5} investigated the connection between the Zagreb indices and the distance $k$-domination number of trees.

Motivated by the above results, in this paper, we study the multiplicative Zagreb indices of trees in terms of distance $k$-domination number. We provide sharp lower bound for $\Pi_1$ and upper bound for $\Pi_2$ in terms of distance $k$-domination number of a tree, and characterize those trees for which the bounds are attained.

\section{Preliminaries}
In this section, we present some propositions, definitions and lemmas which are helpful in our main results.

\begin{lemma} \label{l21} \cite{2} Let $T$ be a tree of order $n>5$ such that $T\ncong P_n, S_n$. Then
\[
\Pi_1(S_n)<\Pi_1(T)<\Pi_1(P_n) \mbox{ and }
 \Pi_2(P_n)<\Pi_2(T)<\Pi_2(S_n).
\]
\end{lemma}

Let $T$ be a tree and $uv\in E(T)$  a non-pendent edge of $T$. Assume that $T-uv=T_1\cup T_2$ with vertex $u\in V(T_1)$ and $v\in V(T_2)$. Let $T_{uv}$ be the tree obtained by identifying the vertex $u$ of $T_1$ and the vertex $v$ of $T_2$ and attaching a pendent vertex $w$ to this vertex.

\begin{lemma} \label{l22} \cite{3}
  Let $T$ be a tree with a non-pendent edge $uv$. Then
\[
\Pi_1(T_{uv})<\Pi_1(T)  \mbox{ and }
\Pi_2(T_{uv})>\Pi_2(T).
\]
\end{lemma}

\begin{lemma} \label{l23} \cite{3}
Let $u$ and $v$ be two distinct vertices in a graph $G$. Let  $u_1,\dots, u_r$ be  pendent neighbors of $u$ and $v_1,\dots, v_t$ pendent neighbors of $v$. Define $G^\prime=G -\{vv_1, \dots, vv_t\} + \{uv_1, \dots, uv_t\}$ and $G^{\prime\prime}=G - \{uu_1, \dots, uu_r\} + \{vu_1, \dots, vu_r\}$. Then
\[
\max\{\Pi_1(G^\prime), \Pi_1(G^{\prime\prime})\}<\Pi_1(G)
\]
and
\[
\min\{\Pi_1(G^\prime), \Pi_1(G^{\prime\prime})\}>\Pi_2(G).
\]
%
\end{lemma}

\begin{lemma} \label{l24} \cite{4} Let $G$ be a connected graph of order $n$ with $n\geq k+1$. Then $\gamma_k(G)\leq \lfloor  \frac n {k+1} \rfloor$.
\end{lemma}

\begin{lemma}  \label{l25} \cite{5} Let $T$ be a tree of order $n$ with maximum degree $\Delta$ and distance $k$-domination number $\gamma_k\geq2$. Then $\Delta \leq n- k\gamma_k$.
\end{lemma}

\begin{lemma} \label{l26} \cite{6} Let $T$ be a tree on $(k+1)n$ vertices. Then $\gamma_k(T)=n$ if and only if one of the following conditions holds:\\
$(1)$ $T$ is any tree on $k+1$ vertices;\\
$(2)$ $T=R \circ k$ for some tree $R$  on $n\geq1$ vertices, where $R \circ k$ is the graph obtained by taking one copy of $R$ and $|V(R)|$ copies of the path $P_{k-1}$ of length $k-1$ and then joining the $ith$ vertex of $R$ to exactly one end vertex in the $ith$ copy of $P_{k-1}$.
\end{lemma}

\begin{lemma} \label{add}
  Let $T$ be a tree of order $n$. Let
  \[
  f(T)=\prod_{w\in V(T)}(d_T(w)+1).
  \]
Then $f(T)\ge  2^{n-1}n$ with equality if and only if $T\cong S_n$.
\end{lemma}

\begin{proof} If $n=1$, it is obviously.
Suppose that $n\ge 2$ and that the result is true for a tree of order $n-1$. Let $u$ be a pendent vertex of $T$, being adjacent to vertex $v$. By induction assumption,
\[
f(T-u)\ge 2^{n-2}(n-1)
\]
i.e.,
\[
d_T(v)\prod_{w\in V(T)\setminus\{u,v\}}(d_T(w)+1) \ge 2^{n-2}(n-1)
\]
with equality if and only if $T-u\cong S_{n-1}$. Now we have
\begin{eqnarray*}
f(T)&=&(d_T(u)+1) (d_T(v)+1) \prod_{w\in V(T)\setminus\{u,v\}}(d_T(w)+1)\\
  &=& 2(d_T(v)+1)\prod_{w\in V(T)\setminus\{u,v\}}(d_T(w)+1)\\
  &\ge & 2(d_T(v)+1)\cdot\frac{2^{n-2}(n-1)}{d_T(v)}\\
  &\ge & 2^{n-1}n.
\end{eqnarray*}
with equalities if and only if $d_T(v)=n-1$ and $T-u\cong S_{n-1}$, i.e., $T\cong S_n$.
\end{proof}

\begin{lemma} \label{addd}
  Let $T$ be a tree of order $n$. Let
  \[
  h(T)=\prod_{w\in V(T)}(d_T(w)+1)^{d_T(w)+1}.
  \]
Then $h(T)\le  4^{n-1}n^n$ with equality if and only if $T\cong S_n$.
\end{lemma}

\begin{proof} If $n=1$, it is obviously.
Suppose that $n\ge 2$ and that the result is true for a tree of order $n-1$. Let $u$ be a pendent vertex of $T$, being adjacent to vertex $v$. By induction assumption,
\[
h(T-u)\le 4^{n-2}(n-1)^{n-1}
\]
i.e.,
\[
(d_T(v))^{d_T(v)}\prod_{w\in V(T)\setminus\{u,v\}}(d_T(w)+1)^{d_T(w)+1}\le 4^{n-2}(n-1)^{n-1}
\]
with equality if and only if $T-u\cong S_{n-1}$. Now we have
\begin{eqnarray*}
h(T)&=&(d_T(u)+1)^{d_T(u)+1} \cdot (d_T(v)+1)^{d_T(u)+1}\cdot \left(\prod_{w\in V(T)\setminus\{u,v\}}(d_T(w)+1)^{d_T(w)+1}\right)\\
  &=& 4(d_T(v)+1)^{d_T(v)+1}\cdot \prod_{w\in V(T)\setminus\{u,v\}}(d_T(w)+1)^{d_T(w)+1}\\
  &\le & 4(d_T(v)+1)^{d_T(v)+1}\cdot \frac{4^{n-2}(n-1)^{n-1}}{(d_T(v))^{d_T(v)}}\\
  &\le & 4^{n-1}n^n.
\end{eqnarray*}
with equalities if and only if $d_T(v)=n-1$ and $T-u\cong S_{n-1}$, i.e., $T\cong S_n$.
\end{proof}

For a graph $G$ with $S\subset V(G)$, let $N_G(S)=\cup_{v\in S} N_G(v)$.

\begin{lemma} \label{129} Let T be a tree with minimum value of first multiplicative Zagreb index or maximum value of second multiplicative zagreb index among all $n$-vertex trees with distance $k$-domination number $\gamma_k$. Let
\[
 B_T=\{x \in V(T)| d_T(w)=1 \mbox{ and } \gamma_k(T-w)=\gamma_k(T)\}.
 \]
 If $B_T\neq \emptyset$, then $|N_T(B_T)|=1$.

\end{lemma}

\begin{proof} Suppose that $|N_T(B_T)|\geq2$, say   $u$, $v\in N_T(B_T)$.  If $u'\notin D$ for some pendent neighbor $u'$ of $u$, then $D\setminus\{u'\}\cup \{u\}$ is a distance $k$-dominating set of $T$. So we may assume that no pendent neighbor of $u$ and $v$ is in $D$. Define $T^\prime=T-vv'+uv'$ and $T^{\prime\prime}=T-uu'+vu'$, where $u'$ ($v'$, respectively) is a pendent neighbor of $u$ ($v$, respectively). Then $\gamma_k(T)=\gamma_k(T^\prime)=\gamma_k(T^{\prime\prime})$.
By Lemma \ref{l23},\[
\max\{\Pi_1(T^\prime), \Pi_1(T^{\prime\prime})\}<\Pi_1(T)
\]
and
\[
\min\{\Pi_1(T^\prime), \Pi_1(T^{\prime\prime})\}>\Pi_2(T).
\]
  a contradiction. Hence  $|N_T(B_T)|=1$.
\end{proof}

\section{Main results}

In this section, we present sharp lower bounds of first multiplicative Zagreb index and upper bounds for second multiplicative zagreb index of a tree of order $n$ with distance $k$-domination number $\gamma_k$.

 A tree is starlike if it contains at most one vertex of degree at least three. Obviously, a starlike tree is either a path or a tree with exactly one vertex of degree at least three. In the latter case, it consists of pendent paths at common vertex.

\begin{defn}
For positive integers $n$, $k$ and $s$ with $n\ge (k+1)s$, define $T_{n,k,s}$ to be  a starlike tree with maximum degree $n-ks$, and if it is not a path, then it has  one pendent path of length $k-1$, $s-1$ pendent paths of length $k$ and $n-(k+1)s$ paths of length $1$.
%
%
\end{defn}


Note that
\[
\Pi_1(T_{n,k,s})=4^{ks-1}(n-ks)^2  \mbox{ and }
\Pi_2(T_{n,k,s})=4^{ks-1}(n-ks)^{n-ks}.
\]

As mentioned earlier, sharp lower bounds on first multiplicative Zagreb index and  upper bounds on the second multiplicative Zagreb index of an $n$-vertex tree with distance $1$-domination number have been given in \cite{1}, so we only consider $k\geq 2$.

\begin{defn}
 If $P=v_0v_1\dots v_d$ is a diametric path of tree $T$ of order $n$, then denote by $T_i$ the component of $T-v_{i-1}v_i-v_iv_{i+1}$ containing $v_i$ for $i=1,\dots,d-1$.
\end{defn}

\begin{defn}
    Denoted by $T^a_{n,k,2}$ the tree formed from the path $P_{2k+2}=v_0v_1\dots v_{2k+1}$ by joining $n-2(k+1)$ pendent vertices to $v_a$, where $a\in\{1,\dots, k\}$.
\end{defn}

For a graph $G$, it is obvious that $\gamma_k(G)\le \gamma_1(G)$ for $k\ge 2$. Note also that $\gamma_1(S_n)=1$. Thus, by Lemma \ref{l21}, we have the following result.

 \begin{theorem}
   Let $T$ be an $n$-vertex tree and $\gamma_k(T)=1$. Then $\Pi_1(T)\geq n^2$ and $\Pi_2(T)\leq n^n$. Either equality holds if and only if $T\cong S_n$.
 \end{theorem}

 \begin{theorem}
 Let  $T$ be a tree of order $n$  with  $\gamma_k(T)=2$. Then
\[
\Pi_1(T)\geq 4^{2k-1}(n-2k)^2 \mbox{ and }
\Pi_2(T)\leq 4^{2k-1}(n-2k)^{n-2k}.
    \]
Either equality holds if and only if $T\cong T^a_{n,k,2}$ with $a\in\{1,\dots,k\}$.
 \end{theorem}

\begin{proof}
Let $T$ be a tree  of order $n$ with distance $k$-domination number $2$ that minimize the first multiplicative Zagreb index and maximize the second multiplicative Zagreb index respectively.

Let  $P=v_0\dots v_d$ be a diametric path of $T$.  If $d\leq 2k$, then $\{v_{\lfloor \frac {d}{2} \rfloor}\}$ is a distance $k$-dominating set of $T$, a contradiction. If $d\geq 2k+2$, for $i\in\{1,\dots,d-2\}$, $T_{v_iv_{i+1}}$ is a tree of order $n$ with distance $k$-domination number $2$, by Lemma \ref{l22}, we have  $\Pi_1(T_{v_iv_{i+1}})<\Pi_1(T)$, and $\Pi_2(T_{v_iv_{i+1}})>\Pi_2(T)$, also a contradiction. Hence $d=2k+1$.

If $T_i$ is not a star for some  $i\in \{1,\dots,2k\}$,  then as above, $T_{v_iv_{i+1}}$ is a tree of order $n$ with distance $k$-domination number $2$ such that $\Pi_1(T_{v_iv_{i+1}})<\Pi_1(T)$, and $\Pi_2(T_{v_iv_{i+1}})>\Pi_2(T)$, a contradiction. Thus each $T_i$ for $i\in \{1,\dots,d-1\}$ is a star with center $v_i$.
Now by Lemma \ref{l23}, for some $a\in \{1,\dots,k\}$, $T\cong T^a_{n,k,2}$.

By direct calculation , $\Pi_1(T)= \Pi_1(T^a_{n,k,2})= 4^{2k-1}(n-2k)^2$ and
  $\Pi_2(T)= \Pi_2(T^a_{n,k,2})= 4^{2k-1}(n-2k)^{n-2k}$ for $a\in \{1,\dots, k\}$.
\end{proof}

\begin{lemma} \label{l31}
  Let $T$ be a tree of order $n$ with distance $k$-domination number  $\gamma_k\geq 3$. If $n=(k+1)\gamma_k$, then
\[
\Pi_1(T)\geq 4^{k\gamma_k-1}(\gamma_k)^2 \mbox{ and }
\Pi_2(T)\leq 4^{k\gamma_k-1}(\gamma_k)^{\gamma_k}.
    \]
Either  equality holds if and only if $T\cong  T_{n,k,\gamma_k}$.
\end{lemma}

\begin{proof} By Lemma \ref{l26}, we have $T=R \circ k$ for some tree $R$  on $\gamma_k$ vertices. For $w\in V(R)$,  $d_R(w)=d_T(w)-1$.
Thus
\begin{eqnarray*}
 \Pi_1(T)&=& \prod_{w\in V(R)}d^2_T(w)  \prod_{z\in V(T)\setminus V(R)}d^2_T(z)\\
 & =&\left(\prod_{w\in V(R)}(d_R(w)+1)^2\right) (2^2)^{(k-1)\gamma_k}\\
 &=&4^{(k-1)\gamma_k}\cdot f^2(R),
\end{eqnarray*}
where $f(R)=\prod\limits_{w\in V(R)}(d_R(w)+1)$. By Lemma \ref{add}, $f^2(R)\ge \left(2^{\gamma_k-1}\cdot \gamma_k\right)^2$ with equality if and only if $R\cong S_{\gamma_k}$. Therefore
\[
 \Pi_1(T)\ge 4^{k\gamma_k-1}(\gamma_k)^2
\]
with equality if and only if  $T=R \circ k$ with $R\cong S_{\gamma_k}$, i.e.,  $T\cong T_{n,k,\gamma_k}$.

Similarly,
\begin{eqnarray*}
 \Pi_2(T)&=&\prod_{w\in V(R)}(d_T(w))^{d_T(w)}\prod_{z\in V(T)\setminus V(R)}(d_T(z))^{d_T(z)}\\
  &=&\left(\prod_{w\in V(R)}(d_R(w)+1)^{d_R(w)+1}\right) 4^{(k-1)\gamma_k}\\
  & = & 4^{(k-1)\gamma_k}h(T),
  \end{eqnarray*}
where $h(T)=\prod\limits_{w\in V(R)}(d_R(w)+1)^{d_R(w)+1}$. By Lemma \ref{addd}, $h(T)\le 4^{\gamma_k-1} (\gamma_k)^{\gamma_k}$ with equality if and only if $R\cong S_{\gamma_k}$. Therefore
\[
 \Pi_2(T)\le 4^{k\gamma_k-1}(\gamma_k)^{\gamma_k}
\]
with equality if and only if  $T=R \circ k$ with $R\cong S_{\gamma_k}$, i.e.,  $T\cong T_{n,k,\gamma_k}$.

\end{proof}

\begin{lemma} \label{l32} Let $T$ be a tree of order $n$ with  $\gamma_k(T)=3$, then
\[
\Pi_1(T)\geq 4^{3k-1}(n-3k)^2 \mbox{ and } \Pi_2(T)\leq 4^{3k-1}(n-3k)^{n-3k}.
\]
Either equality holds if and only if $T\cong  T_{n,k,3}$.
\end{lemma}

\begin{proof}
By Lemmas \ref{l24} and \ref{l31}, we have $n\geq(k+1)\gamma_k$, and the result holds for $n=(k+1)\gamma_k$. We present our proof by induction on $n$. Suppose that $n>3(k+1)$ and the result is true for $n-1$.

Let $P=v_0\dots v_d$ be a diametric path and $D$ be a minimum distance $k$-dominating set of $T$. We claim that $d\geq 2k+2$, for otherwise,  $\{v_k,v_{k+1}\}$ is a distance $k$-dominating set,  a contradiction.
We may choose distance $k$-dominating set $D$ of  cardinality $\gamma_k(G)$ with $\{v_k, v_{d-k}\}\subseteq D$ such that  $(\cup_{a=0}^{k}V(T_a) \setminus \{v_k\})\cap D = \emptyset$ and  $(\cup_{a=d-k}^{d}V(T_a) \setminus \{v_{d-k}\})\cap D = \emptyset$.

%
%
Let   $v_0=w_1,v_d=w_2,\dots,w_m$ be all  the pendent vertices of $T$ and $B_T=\{w_i|1\leq i \leq m, \gamma_k(T-w_i)=\gamma_k(T)\}$.

We claim that $|B_T|\ge 1$.

Suppose that $B_T=\emptyset$. Then  $\gamma_k(T-w_i)=\gamma_k(T)-1$ for $1\leq i \leq m$.

For some $i\in \{1,\dots, k,d-k,\dots, d-1\}$,  if $d_T(v_i)\geq 3$, then $V(T_i)\cap \{w_3,\dots, w_m\}\neq \emptyset$. As $\{v_k,v_{d-k}\}\in D$, we have  $\gamma_k(T-z)=\gamma_k(T)$  for $z\in V(T_i)\cap \{w_3,\dots, w_m\}$, a contradiction. It follows that $d_T(v_i)=2$ for $i\in \{1,\dots, k,d-k,\dots,d-1\}$.

As $d_T(v_1)=2$, we have  $\gamma_k(T-v_0)=\gamma_k(T)-1$. Note that $d_T(v_1, v_{k+1})=k$ and  $(\cup_{a=0}^{k}V(T_a) \setminus \{v_k\})\cap D = \emptyset$. Thus $v_{k+1}\in D$. Similarly, $v_{d-k-1}\in D$.  If $d>2k+2$, then  $v_k,v_{k+1},v_{d-k-1},v_{d-k}$ are all  distinct,  a contradiction to the fact that $\gamma_k(T)=3$. Hence $d=2k+2$ and $D=\{v_k,v_{k+1},v_{d-k}\}$.

If $d_T(v_{k+1})=2$, then $T\cong P_{2k+3}$ and $\{v_k,v_{d-k}\}$ is a distance $k$-dominating set, a contradiction. Then $d_T(v_{k+1})\ge 3$ and thus $m\geq  3$. If $m>3$, then  $\gamma_k(T-w_i)=\gamma_k(T)$ for some $i\in\{1,\dots, m\}$, which is impossible. Hence  $m=3$. Thus $T_{k+1}$ is a path with end vertices $v_{k+1}$ and $w_3$. By the definition of distance $k$-domination number, we have   $d(v_{k+1},w_3)=k$.
It follows that  $|V(T)|=3(k+1)$, which is contradiction. This proves our Claim.

Now by our claim and Lemma \ref{129}, we have $|N_T(B_T)|=1$.


Let $w$ be a pendent vertex such that  $\gamma_k(T-w)=\gamma_k(T)$ and $z$ being a unique vertex adjacent to $x$. Then by lemma \ref{l25}, we have $d_T(z)\leq n-k\gamma_k$.
\[
\frac{d_T(z)}{d_T(z)-1}\ge \frac{n-k\gamma_k}{n-1-k\gamma_k}
\]
with equality if and only if $d_T(z)=n-k\gamma_k$.
Note that
\[
 \Pi_1(T)=\Pi_1(T-w)\cdot \frac{d_T(z)^2}{(d_T(z)-1)^2}.
\]
By induction hypothesis, we have
  \begin{eqnarray*}
  \Pi_1(T)&\geq & 4^{k\gamma_ k-1}(n-1-k\gamma_k)^2\cdot \left(\frac{d_T(z)}{(d_T(z)-1)}\right)^2\\
  & \geq &4^{k\gamma_ k-1}(n-k\gamma_k)^2
  \end{eqnarray*}
with equalities  if and only if $T-w\cong T_{n-1,k,\gamma_k}$ and  $d_T(z)=\Delta= n-k\gamma_k$, i.e., $T\cong  T_{n,k,\gamma_k}$.

Let
\[
g(t)=\frac{x^x}{(x-1)^{x-1}}.
\]
Obviously, $g'(x)=g(x)\log \frac{x}{x-1}>0$ for $x>1$. Thus $g(x)$ is strictly increasing for $x>1$, implying that $g(d_T(z))\le g(n-k\gamma_k)$ with equality if and only if
$d_T(z)=n-k\gamma_k$.
Similarly as above,
\[
\Pi_2(T)=\Pi_2(T-w)\cdot\frac{(d_T(z))^{d_T(z)}}{(d_T(z)-1)^{d_T(z)-1}}=\Pi_2(T-x) g(d_T(z)),
\]
and by induction hypothesis, we have
  \begin{eqnarray*}
    \Pi_2(T) &\leq & 4^{k\gamma_ k-1}(n-1-k\gamma_k)^{n-1-k\gamma_k}g(d_T(z))\\
    &\le & 4^{k\gamma_ k-1}(n-1-k\gamma_k)^{n-1-k\gamma_k}g(n-k\gamma_k)\\
  & = &4^{k\gamma_ k-1}(n-k\gamma_k)^{n-k\gamma_k}
  \end{eqnarray*}
with equalities holds if and only if $T-w\cong T_{n-1,k,\gamma_k}$ and  $d_T(z)=\Delta= n-k\gamma_k$, i.e., $T\cong  T_{n,k,\gamma_k}$.
 \end{proof}

 Let $D$ be a distance $k$-dominating set of a graph $G$.
 Let $N^a_G(v)$ be the set of vertices with distance $a$ from $v$. A vertex $v\in V(G)$ is called a private $k$-neighbor of $u$ with respect to $D$ if $\cup_{a=0}^{k}N^a_G(v)\cap D = \{u\}$, that is $d_G(v,u)\leq k$ and $d_G(v, x)\geq k+1$, for any vertex $x\in D\setminus \{u\}$.

 \begin{theorem}
     Let $T$ be a tree of order $n$ with distance $k$-domination number $\gamma_k\geq3$. Then
\[
\Pi_1(T)\geq 4^{k\gamma_k-1}(n-k\gamma_k)^2 \mbox{ and }
\Pi_2(T)\leq 4^{k\gamma _k-1}(n-k\gamma_k)^{n-k\gamma _k}.
\]
Either equality holds  if and only if $T\cong  T_{n,k,\gamma _k}$.
 \end{theorem}

\begin{proof}
Let $P=v_0\dots v_d$ be a diametric path of $T$. Define $B_T=\{w\in V(T)| d_T (w)=1  \mbox{ and } \gamma_k(T-w)=\gamma_k(T)\}$.
We may choose distance $k$-dominating set $D$ of  cardinality $\gamma_k$ with $\{v_k, v_{d-k}\}\subseteq D$ such that  $(\cup_{a=0}^{k}V(T_a) \setminus \{v_k\})\cap D = \emptyset$ and  $(\cup_{a=d-k}^{d}V(T_a) \setminus \{v_{d-k}\})\cap D = \emptyset$.

If $B_T=\emptyset$ then for $i=0,d,  \gamma_k(T-v_i)=\gamma_k(T)-1$. If $B_T \neq \emptyset$, then by  Lemma~\ref{129}, $|N_T(B_T)|=1$. If $v_0,v_d\in B_T$, then since $d-1>1$, we have $\{v_1, v_{d-1}\}\subseteq |N_T(B_T)|$ and thus $|N_T(B_T)|>1$, a contradiction. Thus, in either case,
we may assume that  $\gamma_k(T-v_0)=\gamma_k(T)-1$, and thus
$\{v_k,v_{k+1},v_{d-k}\}\subseteq D$.

For $i=1, \dots, k$, if some $T_i$ is not a star with center $v_i$, then  applying  Lemma \ref{l22} for a non-pendent edge $e$ of $T_i$ to obtain  a tree $T_e$, we have $\Pi_1(T_e)< \Pi_1(T)$ and
     $\Pi_2(T_e)> \Pi_2(T)$. Thus, we may assume that   $T_i$ is  a star with center $v_i$ for all $i=1, \dots, k$.

If there are at least two vertices, say $v_i$ and $v_j$ with $1\le i<j\le k$,  with degree at least $3$ in $T$, then by Lemma \ref{l23} we may find a tree $T'$ by moving the pendent edges at $v_i$ to $v_j$ or via  such that  $\Pi_1(T')< \Pi_1(T)$ and
     $\Pi_2(T')> \Pi_2(T)$. So we may assume that there is at most one vertex among vertices $v_1, \dots, v_k$ with degree at least $3$.
That is, among vertices $v_1, \dots, v_k$, either each  vertex has degree $2$ or exactly one vertex, say  $v_{i_0}$ has at least one pendent neighbor, where $1\le i_0\le k$.
Let
    \begin{eqnarray*}  T''&=&T- \{v_{i_0}z|z\in N_{T}(v_{i_0})\setminus \{v_{i_0-1},v_{i_0+1}\}\}
    \\
    && + \{v_{k+1}z|z\in N_{T}(v_{i_0})\setminus \{v_{i_0-1},v_{i_0+1}\}\}.
     \end{eqnarray*}
Let $s$ is  number of pendent edges at $v_{i_0}$. If $s=0$, then $T=T''$.

Suppose that $s\ge 1$.
Then for $i=1,\dots,k, d_{T''}(v_i)=2$ and $D$ is minimum distance $k$-dominating set of $T''$. Let $PN_{k,D}(z)$ be the set of all private $k$-neighbors of $z$ with respect to $D$ in $T''$. Then for any $z\in \bigcup_{a=0}^{k}N^a_{T''}(v_k)\setminus \{v_0, \dots,v_k\}$, $d_{T''} (z, v_{k+1})\le k$.  It follows that $D\setminus \{v_k\}$ is a distance $k$-dominating set of the tree $T''-\{v_0,\dots,v_k\}$. Also, $PN_{k,D}(v_{k+1})\subseteq V(T'')\setminus\{v_0,\dots,v_k\}$.  It shows that  $D\setminus \{v_k\}$ is a minimum distance $k$-dominating set of $T''-\{v_0,\dots,v_k\}$. Thus $\gamma_k(T''-\{v_0, \dots,v_k\})=\gamma_k -1$ and $\gamma_k(T''-\{v_0, \dots,v_{k-1}\})=\gamma_k -1$.

Let $\Gamma=\prod_{u\in V(T)\setminus \{v_{i_0}, v_{k+1}\}}d^2_{T}(u)$ and $\Phi=\prod_{u\in V(T)\setminus \{v_{i_0}, v_{k+1}\}}(d_{T}(u))^{d_{T}(u)}$.
Then
\begin{eqnarray*}
    \Pi_1(T)-\Pi_1(T'')&=&\left((2+s)^2 d_{T}^2(v_{k+1})-4 \left(d_{T}(v_{k+1})+s\right)^2\right)\Gamma\\
     &=&s(d_{T}(v_{k+1})-2)\left(s\left(d_{T}(v_{k+1})+2\right)+4d_{T}(v_{k+1})\right)\Gamma\\
     &\ge& 0,
     \end{eqnarray*}
     and thus $\Pi_1(T)\ge \Pi_1(T'')$ with equality if and only if  $d_{T}(v_{k+1})=2$. Also
      \begin{eqnarray*}
      \Pi_2(T'')-\Pi_2(T)&=&\left(4\left(d_{T}(v_{k+1})+s\right)^{d_{T}(v_{k+1})+s}-(s+2)^{s+2}(d_{T}(v_{k+1}))^{d_{T}(v_{k+1})}\right)\Phi\\
      &=& F(s)\Phi\\
      &\ge & 0,
      \end{eqnarray*}
where $F(s)=4\left(d_{T}(v_{k+1})+s\right)^{d_{T}(v_{k+1})+s}-(s+2)^{s+2}(d_{T}(v_{k+1}))^{d_{T}(v_{k+1})}$. It is easy to check that $F(s)$ is an increasing function for $s\ge0$. 
Thus  $\Pi_2(T'')\ge \Pi_2(T)$ with equality if and only if  $d_{T}(v_{k+1})=2$.

Now we have shown that $\Pi_1(T)\ge \Pi_1(T'')$ and $\Pi_2(T'')\ge \Pi_2(T)$ with
either equality if and only if $s=0$ (i.e., $T=T''$) or  $d_{T}(v_{k+1})=2$.

In the following,  we prove that, $\Pi_1(T'')\geq 4^{k\gamma_k-1}(n-k\gamma_k)^2$ and  $\Pi_2(T'')\leq 4^{k\gamma _k-1}(n-k\gamma_k)^{n-k\gamma _k}$ with either equality  if and only if $T''\cong T_{n,k,\gamma _k}$.

By Lemma \ref{l32},  the result holds for $n\ge(k+1)\gamma_k$ and $\gamma_k=3$. Suppose that $\gamma_k\geq4$, and the result is true for  $n\geq(k+1)(\gamma_k-1)$.

Note that  $\gamma_k(T''-\{v_0, \dots,v_k\})=\gamma_k -1$ and $|V(T''-\{v_0, \dots, v_k\})|=n-k-1 > (k+1)(\gamma_k -1)$. Then
\begin{eqnarray*}
        \Pi_1(T'')&=&\Pi_1(T''-\{v_0, \dots,v_k\})\cdot \left(\frac{d_{T''}(v_{k+1})}{(d_{T''}(v_{k+1})-1)}\right)^2\cdot
       \prod_{i=0}^{k}d^2_{T''}(v_i)\\
        &\geq&\Pi_1(T''_{n-k-1,k,\gamma _k-1})\cdot \left(\frac{d_{T''}(v_{k+1})}{(d_{T''}(v_{k+1})-1)}\right)^2\cdot 4^{k}\\
       & = & 4^{k(\gamma_ k-1)-1}\cdot(n-k-1-k(\gamma_k-1))^2\cdot \left(\frac{d_{T''}(v_{k+1})}{(d_{T''}(v_{k+1})-1)}\right)^2\cdot 4^{k}\\
        &\ge &4^{k\gamma_ k-1}\cdot(n-1-k\gamma_k)^2\cdot\frac{(n-k\gamma_k)^2}{(n-1-k\gamma_k)^2}\\
        & =& 4^{k\gamma_ k-1}(n-k\gamma_k)^2
        \end{eqnarray*}
with equalities if and only if $T''-\{v_0,\dots,v_k\}\cong  T''_{n-k-1,k,\gamma _k-1}$, and $d_{T''}(v_{k+1})=\Delta=n-k\gamma_k$. Recall that for $i=1,\dots,k$,  $d_{T''}(v_i)=2$.
Thus $\Pi_1(T'')\ge 4^{k\gamma_ k-1}(n-k\gamma_k)^2$ with equality if and only if $T''\cong T_{n,k,\gamma_k}$.

Similarly, we have
\begin{eqnarray*}
        \Pi_2(T'')&= &\Pi_2(T''-\{v_0, \dots,v_k\}) \cdot\frac{(d_{T''}(v_{k+1}))^{d_{T''}(v_{k+1})}}{(d_{T''}(v_{k+1})-1)^{d_{T''}(v_{k+1})-1}}
         \cdot \prod_{i=0}^{k}(d_{T''}(v_i))^{d_{T}(v_i)}\\
       & \le &\Pi_2(T''_{n-k-1,k,\gamma _k-1})\cdot 4^{k}\cdot
        \frac{(d_{T''}(v_{k+1}))^{d_{T''}(v_{k+1})}}{(d_{T''}(v_{k+1})-1)^{d_{T''}(v_{k+1})-1}}\\
        &=& 4^{k(\gamma_ k-1)-1}(n-k-1-k(\gamma_k-1))^{n-k-1-k(\gamma_k-1)}\cdot 4^{k}\\
       &&  \cdot  \frac{(d_{T''}(v_{k+1}))^{d_{T''}(v_{k+1})}}{(d_{T''}(v_{k+1})-1)^{d_{T''}(v_{k+1})-1}}\\
        & \leq &4^{k\gamma_ k-1}(n-k\gamma_k)^{n-k\gamma_k}
\end{eqnarray*}
with equalities if and only if $T''-\{v_0, \dots,v_k\}\cong  T''_{n-k-1,k,\gamma _k-1}$, and $d_{T''}(v_{k+1})=\Delta=n-k\gamma_k$. Also for $i=1,\dots,k$,  $d_{T''}(v_i)=2$.
Thus $\Pi_2(T'')\le 4^{k\gamma_ k-1}(n-k\gamma_k)^{n-k\gamma_k}$ with equality if and only if $T\cong T_{n,k,\gamma_k}$.

Now  we conclude that $\Pi_1(T)\geq \Pi_1(T'')\geq 4^{k\gamma_ k-1}(n-k\gamma_k)^2$  with equality in the first inequality if and only if $T=T''$ or $T\not\cong T''$ and $d_T(v_{k+1})=2$,
and with equality in the second inequality if and only if $T''\cong  T_{n,k,\gamma_k}$. We show that if $T\not\cong T''$ and $d_T(v_{k+1})=2$, then $T''\not\cong  T_{n,k,\gamma_k}$. Otherwise, say $T''=T_{n,k,\gamma_k}$. As $d_{T''}(v_{k+1})=n-k\gamma_k$, there are $n-(k+1)\gamma_k$ pendent edges at $v_{k+1}$ in $T''$. By the above argument, $d_T(v_{k+1})=\gamma_k>2$, a contradiction. Therefore $\Pi_1(T)\geq 4^{k\gamma_ k-1}(n-k\gamma_k)^2$  with equality if and only if $T\cong T_{n,k,\gamma_k}$.
Similarly, $\Pi_2(T)\leq  4^{k\gamma_ k-1}(n-k\gamma_k)^{n-k\gamma_k}$ with equality if and only if $T\cong T_{n,k,\gamma_k}$.
\end{proof}

\end{document}